\documentclass[12pt,twoside]{amsart}
\usepackage{amsmath, amsthm, amscd, amsfonts, amssymb, graphicx}
\usepackage{enumerate}
\usepackage[colorlinks=true,
linkcolor=blue,
urlcolor=cyan,
citecolor=red]{hyperref}
\usepackage{mathrsfs}
\newtheorem{theorem}{Theorem}[section]

\newtheorem{remark}{Remark}[section]
\newtheorem{definition}{Definition}[section]
\newtheorem{corollary}{Corollary}[section]

\newtheorem{proposition}{Proposition}[section]
\numberwithin{equation}{section}

\begin{document}

\title{Some weighted numerical radius inequalities}
\author{Cristian Conde, Mohammad Sababheh, and Hamid Reza Moradi}
\maketitle

\begin{abstract}
The weighted numerical radius of a Hilbert space operator has been defined recently. This article explores other properties and uses this newly defined numerical radius to obtain several new interesting inequalities for the weighted numerical radius and the numerical radius. In particular, new identities for the numerical radius, further comparisons between the numerical radius and the real and imaginary parts, and some inequalities for sectorial operators will be presented.
\end{abstract}
\pagestyle{myheadings}
\markboth{\centerline {}}
{\centerline {}}
\bigskip
\bigskip
\section{Introduction}
Let $\mathcal{B}(\mathcal{H})$ denote the $C^*-$algebra of all bounded linear operators on a complex Hilbert space $\mathcal{H}$, equipped with the inner product $\left<\cdot,\cdot\right>$. For $A\in\mathcal{B}(\mathcal{H})$, the usual operator norm and the numerical radius of $A$ are defined, respectively, by
$$\|A\|=\sup_{\|x\|=1}\|Ax\|\;{\text{and}}\;\omega(A)=\sup_{\|x\|=1}|\left<Ax,x\right>|.$$
It is well known that the numerical radius is a norm on $\mathcal{B}(\mathcal{H})$, which is equivalent to the operator norm via the relation
\begin{equation}\label{eq_equiv_1}
\frac{1}{2}\|A\|\leq \omega(A)\leq \|A\|.
\end{equation}
Due to the importance of the numerical radius in understanding some geometric aspects of the operator, interest has grown in sharpening the inequalities in \eqref{eq_equiv_1}. We refer the reader to \cite{a2,a3,a1} as a sample of research work in this regard.

Further, some attempts have been made in the literature to extend the notion of the numerical radius. For example, in \cite{O-F} the generalized numerical radius was defined as
\begin{equation}\label{eq_general_w}
\omega_N(A)=\sup_{\theta\in\mathbb{R}}N\left(\mathfrak{R}(e^{i\theta}A)\right),
\end{equation}
where $N$ is a given norm on $\mathcal{B}(\mathcal{H})$. Actually, \eqref{eq_general_w} was inspired from the well known identity \cite{5}:
\begin{equation}
\omega(A)=\sup_{\theta\in\mathbb{R}}\|\mathfrak{R}(e^{i\theta}A)\|.
\end{equation}

 Let $A\in\mathcal{B}(\mathcal{H})$ and let $0\leq t\leq 1$. The weighted real and imaginary parts of $A$ have been defined recently in \cite{sab_aofa} as follows
	\[{{\mathfrak R}_{t}}A=\left( 1-t \right){{A}^{*}}+tA\quad \text{ and }\quad {{\mathfrak I}_{t}}A=\frac{\left( 1-t \right)A-t{{A}^{*}}}{\textup i}.\]
When $t=\frac{1}{2},$ we notice that these reduce back to the usual real and imaginary parts of $A$.

This implies,
	\[{{\mathfrak R}_{t}}A+\textup i{{\mathfrak I}_{t}}A=\left( 1-2t \right){{A}^{*}}+A.\]
In \cite{sab_aofa}, the weighted numerical radius was defined by $w_t(A)=\sup_{\theta}\|\mathfrak{R}_t(e^{i\theta}A)\|.$ This was investigated in this reference, and its relations that generalize those relation of $\omega$ were explored.

This paper defines another weighted numerical radius that seems interesting, as we shall see from its properties. The definition is, in fact, in terms of the numerical radius.

\begin{definition}
Let $A\in\mathcal{B}(\mathcal{H})$ and let $0\leq t\leq 1.$ We define the weighted (or $t-$weighted) numerical radius of $A$ by 
\[{{\omega }_{t}}\left( A \right)=\underset{\left\| x \right\|=1}{\mathop{\underset{x\in \mathcal H}{\mathop{\sup }}\,}}\,\left| \left\langle \left( {{\mathfrak R}_{t}}A+\textup i{{\mathfrak I}_{t}}A \right)x,x \right\rangle  \right|=\omega \left( \left( 1-2t \right){{A}^{*}}+A \right).\]
Similarly, we define the weighted operator norm f $A$ by 
\[{{\left\| A \right\|}_{t}}=\underset{\left\| x \right\|=\left\| y \right\|=1}{\mathop{\underset{x,y\in \mathcal H}{\mathop{\sup }}\,}}\,\left| \left\langle \left( {{\mathfrak R}_{t}}A+\textup i{{\mathfrak I}_{t}}A \right)x,y \right\rangle  \right|=\left\| \left( 1-2t \right){{A}^{*}}+A \right\|.\]
\end{definition}
Before proceeding to the main results, we present the following interesting identity that relates $\mathfrak{I}A$ with $\mathfrak{R}_tA^2$ and $(\mathfrak{R}_tA)^2.$

\begin{proposition}\label{prop_rel_R_I}
Let $A\in\mathcal{B}(\mathcal{H})$ and let $0\leq t\leq 1$. Then 
\begin{equation}\label{010}
\frac{1}{4t\left( 1-t \right)}\left( {{\left( {{\mathfrak R}_{t}}A \right)}^{2}}-{{\mathfrak R}_{t}}{{A}^{2}} \right)={{\left( \mathfrak IA \right)}^{2}}.
\end{equation}
In particular, $$(\mathfrak{R}A)^2-\mathfrak{R}A^2=(\mathfrak{I}A)^2.$$
\end{proposition}
\begin{proof}
Let $0\le t\le 1$. Then
\[\begin{aligned}
  & {{\mathfrak R}_{t}}{{A}^{2}}-{{\left( {{\mathfrak R}_{t}}A \right)}^{2}}+4{{\left( 1-t \right)}^{2}}{{\left( \mathfrak IA \right)}^{2}} \\ 
 & =\left( 1-t \right){{\left( {{A}^{*}} \right)}^{2}}+t{{A}^{2}}-{{\left( \left( 1-t \right){{A}^{*}}+tA \right)}^{2}}-{{\left( 1-t \right)}^{2}}{{\left(A- {{A}^{*}}\right)}^{2}} \\ 
 & =\left( 1-t \right){{\left( {{A}^{*}} \right)}^{2}}+t{{A}^{2}}-{{\left( 1-t \right)}^{2}}{{\left( {{A}^{*}} \right)}^{2}}-{{t}^{2}}{{A}^{2}}-\left( 1-t \right)t\left( {{A}^{*}}A+A{{A}^{*}} \right) \\ 
 &\quad  -{{\left( 1-t \right)}^{2}}{{\left( {{A}^{*}} \right)}^{2}}-{{\left( 1-t \right)}^{2}}{{A}^{2}}+{{\left( 1-t \right)}^{2}}\left( {{A}^{*}}A+A{{A}^{*}} \right) \\ 
 & =\left( 1-t \right)\left( 2t-1 \right){{\left( {{A}^{*}} \right)}^{2}}+\left( 1-t \right)\left( 2t-1 \right){{A}^{2}}-\left( 1-t \right)\left( 2t-1 \right)\left( {{A}^{*}}A+A{{A}^{*}} \right) \\ 
 & =\left( 1-t \right)\left( 2t-1 \right)\left( {{\left( {{A}^{*}} \right)}^{2}}+{{A}^{2}}-\left( {{A}^{*}}A+A{{A}^{*}} \right) \right) \\ 
 & =-4\left( 1-t \right)\left( 2t-1 \right){{\left( \mathfrak IA \right)}^{2}}.  
\end{aligned}\]
Namely,
$$
\frac{1}{4t\left( 1-t \right)}\left( {{\left( {{\mathfrak R}_{t}}A \right)}^{2}}-{{\mathfrak R}_{t}}{{A}^{2}} \right)={{\left( \mathfrak IA \right)}^{2}}.
$$
This proves the first identity, which implies the second identity upon letting $t=\frac{1}{2}.$
\end{proof}
Since $\mathfrak I\textup i{{A}^{*}}=\mathfrak RA$, ${{\mathfrak R}_{t}}{{\left( \textup i{{A}^{*}} \right)}^{2}}=-{{\mathfrak R}_{1-t}}{{A}^{2}}$, and ${{\mathfrak R}_{t}}\textup i{{A}^{*}}={{\mathfrak I}_{t}}A$, we have by \eqref{010}
\[{{\left( {{\mathfrak I}_{t}}A \right)}^{2}}+{{\mathfrak R}_{1-t}}{{A}^{2}}=4t\left( 1-t \right){{\left( \mathfrak RA \right)}^{2}}.\]

As an interesting consequence of Proposition \ref{prop_rel_R_I}, we have the following.

\begin{corollary}
Let $A\in\mathcal{B}(\mathcal{H})$. Then 
\[{{\mathfrak R}_{t}}{{A}^{2}}\le {{\left( {{\mathfrak R}_{t}}A \right)}^{2}}.\]
In particular, \[{{\mathfrak R}}{{A}^{2}}\le {{\left( {{\mathfrak R}}A \right)}^{2}}.\]

\end{corollary}
\begin{proof}
From Proposition \ref{prop_rel_R_I}, we have $$
\frac{1}{4t\left( 1-t \right)}\left( {{\left( {{\mathfrak R}_{t}}A \right)}^{2}}-{{\mathfrak R}_{t}}{{A}^{2}} \right)={{\left( \mathfrak IA \right)}^{2}}.
$$ Since $(\mathfrak{I}A)^2\geq 0$, we infer that 
\[{{\mathfrak R}_{t}}{{A}^{2}}\le {{\left( {{\mathfrak R}_{t}}A \right)}^{2}}.\]
\end{proof}

Integrating the identity
\begin{equation*}
\left( {{\left( {{\mathfrak R}_{t}}A \right)}^{2}}-{{\mathfrak R}_{t}}{{A}^{2}} \right)={4t\left( 1-t \right)}{{\left( \mathfrak IA \right)}^{2}}
\end{equation*}
over $0\leq t\leq 1$ implies the following identity.
\begin{corollary}
Let $A\in\mathcal{B}(\mathcal{H})$. Then
\[\frac{3}{2 }\int\limits_{0}^{1}{{{\left| {{\left( {{\mathfrak R}_{t}}A \right)}^{2}}-{{\mathfrak R}_{t}}{{A}^{2}} \right|}^{\frac{1}{2}}}dt}=\left| \mathfrak IA \right|.\]
\end{corollary}

\section{Main results}
In this section, we present our main results. In the first part, we discuss some properties of $\omega_t$, and then we discuss further inequalities for the numerical radius.

In the following proposition, we list some accessible properties of $\omega_t$ that the reader can check quickly.

\begin{proposition}\label{prop_properties}
Let $A,B\in\mathcal{B}(\mathcal{H})$ and let $0\leq t\leq 1.$ Then

\begin{enumerate}
\item ${{\omega }_{{1}/{2}\;}}\left( A \right)=\omega \left( A \right)$ and ${{\left\| A \right\|}_{{1}/{2}\;}}=\left\| A \right\|$.

\item ${{\omega }_{0}}\left( A \right)=2\left\| \mathfrak RA \right\|$ and ${{\omega }_{1}}\left( A \right)=2\left\| \mathfrak IA \right\|$. 

\item ${\frac{\|A\|_t}{2}\leq}\omega_t(A)\leq \|A\|_t.$

\item $  {{\omega }_{t}}\left( A \right)\le 2\omega \left( A \right).$

\item ${{\omega }_{t}}\left( A+B \right)\le {{\omega }_{t}}\left( A \right)+{{\omega }_{t}}\left( B \right).$

\item The function
$f\left( t \right)={{\omega }_{t}}\left( A \right)$
is convex on the interval $\left[ 0,1 \right]$.
\end{enumerate}
\end{proposition}

Another interesting property of $\omega_t$ is that it is self-adjoint.
\begin{proposition}\label{16}
Let $A,B\in \mathcal B\left( \mathcal H \right)$. Then for any $0\le t\le 1$,
\begin{equation*}
{{\omega }_{t}}\left( {{A}^{*}} \right)={{\omega }_{t}}\left( A \right).
\end{equation*}
\end{proposition}
\begin{proof}
For any $0\le t\le 1$,
\begin{equation}\label{10}
\begin{aligned}
   {{\omega }_{t}}\left( {{A}^{*}} \right)&=\omega \left( \left( 1-2t \right)A+{{A}^{*}} \right) \\ 
 & =\omega \left( \left( 1-2t \right)\left( \mathfrak RA+\textup i\mathfrak IA \right)+\mathfrak RA-\textup i\mathfrak IA \right) \\ 
 & =2\omega \left( \left( 1-t \right)\mathfrak RA-t\textup i\mathfrak IA \right) \\ 
 & =2\omega \left( \left( 1-t \right)\mathfrak RA+t\textup i\mathfrak IA \right),
\end{aligned}
\end{equation}
where the last identity follows because $\omega(T)=\omega(T^*)$ for any $T\in\mathcal{B}(\mathcal{H}).$
Moreover,
\begin{equation}\label{8}
\begin{aligned}
   {{\omega }_{t}}\left( A \right)&=\omega \left( \left( 1-2t \right){{A}^{*}}+A \right) \\ 
 & =\omega \left( \left( 1-2t \right)\left( \mathfrak RA-\textup i\mathfrak IA \right)+\mathfrak RA+\textup i\mathfrak IA \right) \\ 
 & =2\omega \left( \left( 1-t \right)\mathfrak RA+t\textup i\mathfrak IA \right).  
\end{aligned}
\end{equation}
The identities \eqref{10} and \eqref{8} together indicate that 
\begin{equation*}
{{\omega }_{t}}\left( {{A}^{*}} \right)={{\omega }_{t}}\left( A \right),
\end{equation*}
as desired.
\end{proof}

The following result explicitly finds some comparisons between $\omega_t$ and $\omega$.
\begin{theorem}\label{3}
Let $A\in \mathcal B\left( \mathcal H \right)$. Then for any $0\le t\le 1$,
\[\frac{{{\omega }_{r}}\left( A \right)}{2R}\le \omega \left( A \right)\le \frac{{{\omega }_{R}}\left( A \right)}{2r},\]
where $r=\min \left\{ t,1-t \right\}$ and $R=\max \left\{ t,1-t \right\}$. 
\end{theorem}
\begin{proof}
If $0\le t\le {1}/{2}\;$, then
\[\begin{aligned}
   {{\omega }_{t}}\left( A \right)&=\omega \left( \left( 1-2t \right){{A}^{*}}+A \right) \\ 
 & \le \left( 1-2t \right)\omega \left( {{A}^{*}} \right)+\omega \left( A \right) \\ 
 & =2\left( 1-t \right)\omega \left( A \right).  
\end{aligned}\]
For ${1}/{2}\;\le t\le 1$, 
\[\begin{aligned}
   \omega((1-2t)A^*+A)&=\omega(A-(2t-1)A^*)\\
   &\geq w(A)-(2t-1)w(A^*)\\
   &=w(A)-(2t-1)w(A^*)\\
   &=2(1-t)w(A).  
\end{aligned}\]
By combining the above two inequalities, we reach the desired result.
\end{proof}
Integrating the inequalities in Theorem \ref{3} over the corresponding $t$ values implies the following bounds.
\begin{corollary}
Let $A\in \mathcal B\left( \mathcal H \right)$. Then
\[\int\limits_{0}^{\frac{1}{2}}{{{\omega }_{t}}\left( A \right)}dt\le \frac{3}{4}\omega \left( A \right)\quad\text{ and }\quad \frac{1}{4}\omega \left( A \right)\le \int\limits_{\frac{1}{2}}^{1}{{{\omega }_{t}}\left( A \right)}dt.\]
\end{corollary}

It is well known that $\|\mathfrak{R}A\|\leq \omega(A)$ for any $A\in\mathcal{B}(\mathcal{H})$. In the following, we find an interesting refinement of this inequality in terms of $\omega_t.$
\begin{corollary}\label{15}
Let $A\in \mathcal B\left( \mathcal H \right)$. Then for any $0\le t\le 1$, 
\[\left\| \mathfrak RA \right\|\le \frac{{{\omega }_{r}}\left( A \right)}{2R}\le \omega \left( A \right),\]
where $r=\min \left\{ t,1-t \right\}$ and $R=\max \left\{ t,1-t \right\}$. In particular,
\[\left\| \mathfrak RA \right\|\le \frac{4}{3}\int\limits_{0}^{\frac{1}{2}}{{{\omega }_{t}}\left( A \right)dt}\le \omega \left( A \right).\]
\end{corollary}
\begin{proof}
We have
\[\begin{aligned}
  2\left( 1-t \right)\omega \left( A+{{A}^{*}} \right)& ={{\omega }_{t}}\left( A+{{A}^{*}} \right) \\ 
 & \le {{\omega }_{t}}\left( A \right)+{{\omega }_{t}}\left( {{A}^{*}} \right) \quad \text{(by Propoition \ref{prop_properties})}\\ 
 & =2{{\omega }_{t}}\left( A \right) \quad \text{(by Proposition \ref{16})}\\ 
 & \le 4\left( 1-t \right)\omega \left( A \right)\quad \text{(by Theorem \ref{3})}.  
\end{aligned}\]
This completes the proof of the first part, noting that $\omega(A+A^*)=2\|\mathfrak{R}A\|$. For the second part, we integrate the first part of the result over $0\leq t\leq \frac{1}{2}.$
\end{proof}

\begin{remark}\label{rem1}
Replacing $A$ by $\textup i{{A}^{*}}$, in Corollary \ref{15}, we get for any $0\le t\le {1}/{2}\;$,
\begin{equation}\label{11}
\left\| \mathfrak IA \right\|\le \frac{\omega \left( \left( 1-2t \right)A-{{A}^{*}} \right)}{2\left( 1-t \right)}\le \omega \left( A \right),
\end{equation}
since
\[{{\omega }_{t}}\left( \textup i{{A}^{*}} \right)=2\omega \left( \left( 1-t \right)\mathfrak IA+t\textup i\mathfrak RA \right)=\omega \left( \left( 1-2t \right)A-{{A}^{*}} \right).\]
In particular,\[\left\| \mathfrak IA \right\|\le \frac{4}{3}\int\limits_{0}^{\frac{1}{2}}{\omega \left( \left( 1-2t \right)A-{{A}^{*}} \right)dt}\le \omega \left( A \right).\]
\end{remark}
Combining Corollary \ref{15} and Remark \ref{rem1}, we obtain the following refinement of the first inequality in \eqref{eq_equiv_1}.
  \begin{proposition}
  Let $A\in \mathcal B\left( \mathcal H \right)$, then
  \begin{equation*}
  \frac{\|A\|}{2}\leq \frac12\left( \left\|\mathfrak RA \right\|+\left\| \mathfrak IA \right\|\right)\le \frac{2}{3}\left(\int\limits_{0}^{\frac{1}{2}}{[{{\omega }_{t}}\left( A \right)+\omega \left( \left( 1-2t \right)A-{{A}^{*}} \right)]dt}\right)\le \omega \left( A \right).
  \end{equation*}
  \end{proposition}

Finding new identities for known quantities can help obtain new results or forms that are difficult to obtain from the original forms. The following is a new identity of $\omega$ in terms of $\omega_t$. The significance of this result is two-folded. First, it extends the identity $\omega(A)=\sup_{\theta}\|\mathfrak{R}(e^{i\theta}A)\|$, and second it provides a formula that involves $t$ within its terms, but with an independent result of $t$, namely $\omega(A)$.

\begin{theorem}
Let $A\in \mathcal B\left( \mathcal H \right)$. Then for any $0\le t\le 1$, \begin{equation*}
\omega \left( A \right)=\underset{\theta \in \mathbb{R}}{\mathop{\sup }}\,\left\{ \frac{{{\omega }_{r}}\left( {{e}^{\textup i\theta }}A \right)}{2R} \right\},
\end{equation*}
where $r=\min \left\{ t,1-t \right\}$ and $R=\max \left\{ t,1-t \right\}$.
\end{theorem}
\begin{proof}
Replacing $A$ by ${{e}^{\textup i\theta }}A$ in Corollary \ref{15}, we get
\[\left\| \mathfrak R{{e}^{\textup i\theta }}A \right\|\le \frac{{{\omega }_{r}}\left( {{e}^{\textup i\theta }}A \right)}{2 R }\le \omega \left( {{e}^{\textup i\theta }}A \right)=\left| {{e}^{\textup i\theta }} \right|\omega \left( A \right)=\omega \left( A \right).\]
Taking the supremum over $\theta$, we get 
\[\underset{\theta \in \mathbb{R}}{\mathop{\sup }}\,\left\| \mathfrak R{{e}^{\textup i\theta }}A \right\|\le \underset{\theta \in \mathbb{R}}{\mathop{\sup }}\,\left\{ \frac{{{\omega }_{t}}\left( {{e}^{\textup i\theta }}A \right)}{2\left( 1-t \right)} \right\}\le \omega \left( A \right).\]
Noting that $\underset{\theta \in \mathbb{R}}{\mathop{\sup }}\,\left\| {{\operatorname{\mathfrak Re}}^{\textup i\theta }}A \right\|=\omega \left( A \right)$, by \cite[(2.3)]{5}, we get the desired result.
\end{proof}

In the first inequality of \eqref{eq_equiv_1}, we have $\frac{1}{2}\|A\|\leq\omega(A)$. It is astonishing that this inequality is a special case of a more general form in terms of $\|A\|_t$ and $\omega_t(A)$, as follows.
\begin{theorem}
Let $A\in \mathcal B\left( \mathcal H \right)$. Then for any $0\le t\le 1$,
\[\frac{1}{4R}{{\left\| A \right\|}_{r}}\le \omega \left( A \right),\]
where $r=\min \left\{ t,1-t \right\}$ and $R=\max \left\{ t,1-t \right\}$. In particular, $\frac{1}{2}\|A\|\leq \omega(A).$
\end{theorem}
\begin{proof}
If $0\leq t\leq \frac{1}{2},$ we have
\[\begin{aligned}
   {{\left\| A \right\|}_{t}}&=\left\| \left( 1-2t \right){{A}^{*}}+A \right\| \\ 
 & \le \left( 1-2t \right)\left\| {{A}^{*}} \right\|+\left\| A \right\| \\ 
 & \le 4\left( 1-t \right)\omega \left( A \right),
\end{aligned}\]
where the last inequality follows from the facts
\begin{equation*}
\left\| A \right\|\le 2\omega \left( A \right)\text{ and }\left\| {{A}^{*}} \right\|\le 2\omega \left( {{A}^{*}} \right)=2\omega \left( A \right).
\end{equation*}
This proves the desired inequality for $0\leq t\leq \frac{1}{2}.$ On the other hand, if $\frac{1}{2}\;\le t\le 1$, then
\[\begin{aligned}
   {{\left\| A \right\|}_{1-t}}&=2\left\| t\mathfrak RA+\left( 1-t \right)\textup i\mathfrak IA \right\| \\ 
 & =\left\| \left( 2t-1 \right){{A}^{*}}+A \right\| \\ 
 & \le \left( 2t-1 \right)\left\| {{A}^{*}} \right\|+\left\| A \right\| \\ 
 & \le 4t\omega \left( A \right).  
\end{aligned}\]
This completes the proof.
\end{proof}

Employing the  refinement of Hermite-Hadamard inequality \cite{EF},
 $$f\left(\frac{1}{2}\right)\leq l(\lambda)\leq \int_0^1 f(t)dt\leq L(\lambda) \leq \frac{f(0)+f(1)}{2}$$
	for convex function $f\text{:}\left[ 0,1 \right]\to \mathbb{R}$ and $\lambda \in [0,1]$, with
	$$
	l(\lambda)=\lambda f\left(\frac{\lambda  }{2}\right)+(1-\lambda) f\left(\frac{1+\lambda}{2}\right)
	$$
	and
	$$
	L(\lambda)=\frac12 (f(\lambda  )+\lambda f(0)+(1-\lambda)f(1)),
	$$
we have the following result, which refines the inequality $\omega(A)\leq \|\mathfrak{R}A\|+\|\mathfrak{I}A\|,$ noting that the function $f(t)=\omega_t(A)$ is convex.
\begin{corollary}\label{conv2}
Let $A\in \mathcal B\left( \mathcal H \right)$. Then
\[\begin{aligned}\label{wFar}
\omega \left( A \right)&\le \lambda \omega_{\frac{\lambda}{2}}(A)+(1-\lambda)\omega_{\frac{1+\lambda}{2}}(A)\nonumber \\
&\le \int\limits_{0}^{1}{{{\omega }_{t}}\left( A \right)dt}\nonumber\\
&\le \frac12 \omega_{\lambda}(A)+\lambda \left\| \mathfrak RA \right\|+(1-\lambda)\left\| \mathfrak IA \right\| \nonumber\\
 &\le\left\| \mathfrak RA \right\|+\left\| \mathfrak IA \right\|,
 \end{aligned}\]
for all $\lambda \in [0, 1].$
\end{corollary}
From Corollary \ref{conv2}, we obtain that
\[\begin{aligned}
	\omega \left( A \right)&\le \sup_{\lambda\in [0,1]}\left[\lambda \omega_{\frac{\lambda}{2}}(A)+(1-\lambda)\omega_{\frac{1+\lambda}{2}}(A)\right]\\
	&\le \int\limits_{0}^{1}{{{\omega }_{t}}\left( A \right)dt}\nonumber\\
	&\le \inf_{\lambda\in [0,1]}\left[\frac12 \omega_{\lambda}(A)+\lambda \left\| \mathfrak RA \right\|+(1-\lambda)\left\| \mathfrak IA \right\|\right]\\ 
	&\le\left\| \mathfrak RA \right\|+\left\| \mathfrak IA \right\|.
\end{aligned}\]

\begin{remark}
	If $A$ is a non-zero self-adjoint operator, then for any $\lambda \in [0, 1]$ holds
\[\begin{aligned}
	\|A\|&=\lambda \omega_{\frac{\lambda}{2}}(A)+(1-\lambda)\omega_{\frac{1+\lambda}{2}}(A)\\
	&=\int\limits_{0}^{1}{{{\omega }_{t}}\left( A \right)dt}\nonumber \\
	&= \frac12 \omega_{\lambda}(A)+\lambda \left\| A \right\|.\nonumber\
\end{aligned}\]
	In particular, we have that 
	$$
	\|A\|=\frac{\omega_{\lambda}(A)}{2(1-\lambda)}.
	$$
\end{remark}

As we have seen, the inequality $\omega(A)\leq\|\mathfrak{R}A\|+\|\mathfrak{I}A\|,$ for any $A\in\mathcal{B}(\mathcal{H})$. In the following, we present upper and lower bounds of the difference between the two sides of this inequality in terms of $\omega_t.$ To achieve this, we need to remind the reader of the following refinement and reverse of the celebrated Jensen inequality \cite{17}: If $f:\left[ 0,1 \right]\to \mathbb{R}$ is a convex function, then for any $0\le t\le 1$,
\[\begin{aligned}
   \frac{\left( 1-t \right)f\left( 0 \right)+tf\left( 1 \right)-f\left( t \right)}{2R}&\le \frac{f\left( 0 \right)+f\left( 1 \right)}{2}-f\left( \frac{1}{2} \right) \\ 
 & \le \frac{\left( 1-t \right)f\left( 0 \right)+tf\left( 1 \right)-f\left( t \right)}{2r},  
\end{aligned}\]
where $r=\min \left\{ t,1-t \right\}$ and $R=\max \left\{ t,1-t \right\}$. Employing this and part (6) of Proposition \ref{prop_properties}, we can reach the following result. Notice that the first inequality can be obtained from Corollary \ref{conv2}.
\begin{corollary}
Let $A\in \mathcal B\left( \mathcal H \right)$. Then for any $0\le t\le 1$,
\[\begin{aligned}
   \frac{1}{R}\left( \left( 1-t \right)\left\| \mathfrak RA \right\|+t\left\| \mathfrak IA \right\|-\frac{1}{2}{{\omega }_{t}}\left( A \right) \right)&\le \left\| \mathfrak RA \right\|+\left\| \mathfrak IA \right\|-\omega \left( A \right) \\ 
 & \le \frac{1}{r}\left( \left( 1-t \right)\left\| \mathfrak RA \right\|+t\left\| \mathfrak IA \right\|-\frac{1}{2}{{\omega }_{t}}\left( A \right) \right),  
\end{aligned}\]
where $r=\min \left\{ t,1-t \right\}$ and $R=\max \left\{ t,1-t \right\}$. In particular,
\[2\int\limits_{0}^{1}{{{\omega }_{t}}\left( A \right)dt}-\left( \left\| \mathfrak RA \right\|+\left\| \mathfrak IA \right\| \right)\le \omega \left( A \right)\le \frac{2}{3}\int\limits_{0}^{1}{{{\omega }_{t}}\left( A \right)dt}+\frac{1}{3}\left( \left\| \mathfrak RA \right\|+\left\| \mathfrak IA \right\| \right).\]
\end{corollary}

On the other hand, the following result presents an explicit comparison between  $\omega_t(A)$ and $\|(1-t)\mathfrak{R}A+t\mathfrak{I}A\|.$
\begin{proposition}\label{9}
Let $A\in \mathcal B\left( \mathcal H \right)$. Then for any $0\le t\le 1$,
\[\sqrt{2}\left\| \left( 1-t \right)\mathfrak RA+t\mathfrak IA \right\|\le {{\omega }_{t}}\left( A \right).\]
In particular, $\frac{1}{\sqrt{2}}\|\mathfrak{R}A+\mathfrak{I}A\|\leq \omega(A).$
\end{proposition}
\begin{proof}
For any vector $x\in \mathcal H$, we have
\[\begin{aligned}
   \left| \left\langle \left( \left( 1-2t \right){{A}^{*}}+A \right)x,x \right\rangle  \right|&=2\left| \left( 1-t \right)\left\langle \mathfrak RAx,x \right\rangle +t\left\langle \textup i\mathfrak IAx,x \right\rangle  \right| \\ 
 & =2\left| \left\langle \left( 1-t \right)\mathfrak RAx,x \right\rangle +\textup i\left\langle t\mathfrak IAx,x \right\rangle  \right| \\ 
 & \ge \frac{2}{\sqrt{2}}\left| \left( 1-t \right)\left\langle \mathfrak RAx,x \right\rangle +t\left\langle \mathfrak IAx,x \right\rangle  \right| \\ 
 & = \sqrt{2}\left| \left\langle \left( \left( 1-t \right)\mathfrak RA+t\mathfrak IA \right)x,x \right\rangle  \right|,
\end{aligned}\]
where in the above computations, we have used $|a+ib|\geq \frac{1}{\sqrt{2}}|a+b|$ for the real numbers of $a,b$. This completes the proof.
\end{proof}

In \cite{07}, it was shown that for any $A\in\mathcal{B}(\mathcal{H})$, one has the inequality
\begin{equation}\label{ineq_kitt_2}
\omega(A)^2\leq \frac{1}{2}\|\;|A|^2+|A^*|^2\|,
\end{equation}
as a refinement of the inequality $\omega(A)\leq \|A\|.$ In the following theorem, we find the $\omega_t$ version of this inequality, which retrieves \eqref{ineq_kitt_2} when $t=\frac{1}{2}.$
\begin{theorem}\label{thm_tR_1}
Let $A\in \mathcal B\left( \mathcal H \right)$. Then for any $0\le t\le 1$,
\[\left\| \sqrt{{{\left( 1-t \right)}^{3}}}\mathfrak RA\pm \sqrt{{{t}^{3}}}\mathfrak IA \right\|\le \frac{1}{2}{{\omega }_{t}}\left( A \right)\le {{\left\| {{\left( 1-t \right)}^{2}}{{\left( \mathfrak RA \right)}^{2}}+{{t}^{2}}{{\left( \mathfrak IA \right)}^{2}} \right\|}^{\frac{1}{2}}}.\]
In particular,
\[\frac{1}{\sqrt{2}}\left\| \mathfrak RA\pm \mathfrak IA \right\|\le \omega \left( A \right)\le {{\left\| {{\left( \mathfrak RA \right)}^{2}}+{{\left( \mathfrak IA \right)}^{2}} \right\|}^{\frac{1}{2}}}.\]
\end{theorem}
\begin{proof}
By \eqref{8}, for any vector $x\in \mathcal H$, we have
\[{{\left| \left\langle \left( \left( 1-2t \right){{A}^{*}}+A \right)x,x \right\rangle  \right|}^{2}}=4\left( {{\left\langle \left( 1-t \right)\mathfrak RAx,x \right\rangle }^{2}}+{{\left\langle t\mathfrak IAx,x \right\rangle }^{2}} \right).\]
Now, it follows from the convexity of the function $f\left( t \right)={{t}^{2}}$,
\[\begin{aligned}
   {{\left| \left\langle \left( \left( 1-2t \right){{A}^{*}}+A \right)x,x \right\rangle  \right|}^{2}}&=4\left( {{\left\langle \left( 1-t \right)\mathfrak RAx,x \right\rangle }^{2}}+{{\left\langle t\mathfrak IAx,x \right\rangle }^{2}} \right) \\ 
 & =4\left( \left( 1-t \right){{\left\langle \sqrt{1-t}\mathfrak RAx,x \right\rangle }^{2}}+t{{\left\langle \sqrt{t}\mathfrak IAx,x \right\rangle }^{2}} \right) \\ 
 & \ge 4{{\left( \sqrt{{{\left( 1-t \right)}^{3}}}\left| \left\langle \mathfrak RAx,x \right\rangle  \right|+\sqrt{{{t}^{3}}}\left| \left\langle \mathfrak IAx,x \right\rangle  \right| \right)}^{2}} \\ 
 & \ge 4{{\left| \left\langle \left( \sqrt{{{\left( 1-t \right)}^{3}}}\mathfrak RA\pm \sqrt{{{t}^{3}}}\mathfrak IA \right)x,x \right\rangle  \right|}^{2}}.  
\end{aligned}\]
Therefore,
\begin{equation}\label{1}
4{{\left\| \sqrt{{{\left( 1-t \right)}^{3}}}\mathfrak RA\pm \sqrt{{{t}^{3}}}\mathfrak IA \right\|}^{2}}\le \omega _{t}^{2}\left( A \right).
\end{equation}
This proves the first inequality.
On the other hand,
\[\begin{aligned}
   {{\left| \left\langle \left( \left( 1-2t \right){{A}^{*}}+A \right)x,x \right\rangle  \right|}^{2}}&=4\left( {{\left\langle \left( 1-t \right)\mathfrak RAx,x \right\rangle }^{2}}+{{\left\langle t\mathfrak IAx,x \right\rangle }^{2}} \right) \\ 
 & \le 4\left( {{\left( 1-t \right)}^{2}}{{\left\| \mathfrak RAx \right\|}^{2}}+{{t}^{2}}{{\left\| \mathfrak IAx \right\|}^{2}} \right) \\ 
 & =4\left( {{\left( 1-t \right)}^{2}}\left\langle \mathfrak RAx,\mathfrak RAx \right\rangle +{{t}^{2}}\left\langle \mathfrak IAx,\mathfrak IAx \right\rangle  \right) \\ 
 & =4\left\langle \left( {{\left( 1-t \right)}^{2}}{{\left( \mathfrak RA \right)}^{2}}+{{t}^{2}}{{\left( \mathfrak IA \right)}^{2}} \right)x,x \right\rangle.
\end{aligned}\]
So,
\begin{equation*}
\omega _{t}^{2}\left( A \right)\le 4\left\| {{\left( 1-t \right)}^{2}}{{\left( \mathfrak RA \right)}^{2}}+{{t}^{2}}{{\left( \mathfrak IA \right)}^{2}} \right\|.
\end{equation*}
This completes the proof.
\end{proof}

The fact that Theorem \ref{thm_tR_1} retrieves \eqref{ineq_kitt_2} follows by noting $|A|^2+|A^*|^2=2\left((\mathfrak{R}A)^2+(\mathfrak{I}A)^2\right).$

\begin{corollary}\label{cor_R_t}
Let $A\in \mathcal B\left( \mathcal H \right)$. Then for any $0\le t\le 1$,
\[2\sqrt{{{\left( 1-t \right)}^{3}}}\left\| \mathfrak RA \right\|+\left|\, \left\| \sqrt{{{\left( 1-t \right)}^{3}}}\mathfrak RA+\sqrt{{{t}^{3}}}\mathfrak IA \right\|-\left\| \sqrt{{{\left( 1-t \right)}^{3}}}\mathfrak RA-\sqrt{{{t}^{3}}}\mathfrak IA \right\| \,\right|\le {{\omega }_{t}}\left( A \right).\]
In particular,
\[\frac{1}{\sqrt{2}}\|\mathfrak{R}A\|+\frac{1}{2\sqrt{2}}\left|\;\|\mathfrak{R}A+\mathfrak{I}A\|-\|\mathfrak{R}A-\mathfrak{I}A\|\;\right|\leq \omega(A).\]
\end{corollary}
\begin{proof}
From \eqref{1}, we infer that
\[\begin{aligned}
  & 2\sqrt{{{\left( 1-t \right)}^{3}}}\left\| \mathfrak RA \right\|+\left|\, \left\| \sqrt{{{\left( 1-t \right)}^{3}}}\mathfrak RA+\sqrt{{{t}^{3}}}\mathfrak IA \right\|-\left\| \sqrt{{{\left( 1-t \right)}^{3}}}\mathfrak RA-\sqrt{{{t}^{3}}}\mathfrak IA \right\|\, \right| \\ 
 & \le \left\| \sqrt{{{\left( 1-t \right)}^{3}}}\mathfrak RA+\sqrt{{{t}^{3}}}\mathfrak IA \right\|+\left\| \sqrt{{{\left( 1-t \right)}^{3}}}\mathfrak RA-\sqrt{{{t}^{3}}}\mathfrak IA \right\| \\ 
 &\qquad +\left| \,\left\| \sqrt{{{\left( 1-t \right)}^{3}}}\mathfrak RA+\sqrt{{{t}^{3}}}\mathfrak IA \right\|-\left\| \sqrt{{{\left( 1-t \right)}^{3}}}\mathfrak RA-\sqrt{{{t}^{3}}}\mathfrak IA \right\|\, \right| \\ 
 & =2\max \left\{ \left\| \sqrt{{{\left( 1-t \right)}^{3}}}\mathfrak RA+\sqrt{{{t}^{3}}}\mathfrak IA \right\|,\left\| \sqrt{{{\left( 1-t \right)}^{3}}}\mathfrak RA-\sqrt{{{t}^{3}}}\mathfrak IA \right\| \right\} \\ 
 & =2\left\| \sqrt{{{\left( 1-t \right)}^{3}}}\mathfrak RA\pm \sqrt{{{t}^{3}}}\mathfrak IA \right\| \\ 
 & \le {{\omega }_{t}}\left( A \right). \\ 
\end{aligned}\]
Hence
\[2\sqrt{{{\left( 1-t \right)}^{3}}}\left\| \mathfrak RA \right\|+\left|\, \left\| \sqrt{{{\left( 1-t \right)}^{3}}}\mathfrak RA+\sqrt{{{t}^{3}}}\mathfrak IA \right\|-\left\| \sqrt{{{\left( 1-t \right)}^{3}}}\mathfrak RA-\sqrt{{{t}^{3}}}\mathfrak IA \right\| \,\right|\le {{\omega }_{t}}\left( A \right),\]
as desired.
\end{proof}

\section{Inequalities for accretive operators}

In this section, we find related inequalities for accretive and accretive-dissipative operators. We recall here that an operator $A\in\mathcal{B}(\mathcal{H})$ is said to accretive if $\mathfrak{R}A>0$, and it is called dissipative if $\mathfrak{I}A> 0$. If $A$ is both accretive and dissipative, it is called accretive-dissipative. 
\begin{corollary}\label{ac-di}
Let $A\in \mathcal B\left( \mathcal H \right)$ be accretive-dissipative. Then for any $0\le t\le 1$,
\[\frac{1}{\sqrt{2}}{{\left\| A \right\|}_{t}}\le {{\omega }_{t}}\left( A \right).\]
In particular,
\begin{equation}\label{07}
\frac{1}{\sqrt{2}}\left\| A \right\|\le \omega \left( A \right).
\end{equation}
\end{corollary}
\begin{proof}
From \cite[Proposition 3.8]{4}, we know that if $S,T$ are two positive operators, then
\begin{equation}\label{13}
\left\| S+\textup iT \right\|\le \left\| S+T \right\|.
\end{equation}
Now, the desired inequality follows from the inequality \eqref{13} and Proposition \ref{9}, since
\[{{\left\| A \right\|}_{t}}=2\left\| \left( 1-t \right)\mathfrak RA+\textup i t\mathfrak IA \right\|.\]
This completes the proof.
\end{proof}

We remark here that the inequality \eqref{07} has been given in \cite[Theorem 2.3]{3}.

Let $A\in \mathcal B\left( \mathcal H \right)$,  by a simple application of
	the triangle inequality and the arithmetic-geometric mean inequality, we have
	\[\|A\|^2\leq 2(\|\mathfrak RA\|^2+\|\mathfrak IA\|^2 ).\]
	However, it is known (see \cite{Mir}) that if $A$ is accretive-dissipative, then we get the following refinement
	\begin{equation}\label{acc-diss} \|A\|^2\leq \|\mathfrak RA\|^2+\|\mathfrak IA\|^2.
	\end{equation}
Combining Corollary \ref{ac-di} and inequality \eqref{acc-diss}, we obtain the following statement.
\begin{proposition}
Let $A\in \mathcal B\left( \mathcal H \right)$ be accretive-dissipative. Then, for any $0\le t\le 1$,
\begin{equation}
\frac18 \|A\|^2_t\leq \frac 14\omega_t^2(A)\leq \omega^2(A)\leq \|A\|^2\leq \|\mathfrak RA\|^2+\|\mathfrak IA\|^2.
\end{equation}
\end{proposition}
\begin{proof}
	We have
	\[\begin{aligned}
	\frac12\|A\|^2_t& \leq {{\omega }^2_{t}}(A) \quad \text{(by Corollary \ref{ac-di})}\\& \leq 4 {\omega^2 }\left( A \right) \quad \text{(by Proposition \ref{prop_properties})}\\ 
	& \leq 4\|A\|^2 \quad \text{(by Proposition \ref{eq_equiv_1})}\\ 
	& \leq 4(\|\mathfrak RA\|^2+\|\mathfrak IA\|^2)\quad \text{(by \eqref{acc-diss})}. 
	\end{aligned}\]
	\end{proof}

A subclass of accretive operators is the so called sectorial operators.  The numerical range of an operator $A\in\mathcal{B}(\mathcal{H})$ is defined by the set $$W(A)=\{\left<Ax,x\right>:x\in\mathcal{H},\|x\|=1\}.$$ 
Let $0\leq \theta<\frac{\pi}{2},$ and let $$S_{\theta}=\{z\in\mathbb{C}:\mathfrak{R}z>0, |\mathfrak{I}z|\leq\tan\theta\;\mathfrak{R}z\}.$$
This presents a sector in the complex plane $\mathbb{C}$. If $A\in\mathcal{B}(\mathcal{H})$ is such that $W(A)\subset S_{\theta}$ for some $0\leq \theta<\frac{\pi}{2},$ we say that $A$ is sectorial and we write $A\in\Pi_{\theta}.$ With this notation, it is implicitly understood that $0\leq \theta<\frac{\pi}{2}.$
\begin{remark}
From \cite{6}, we know that if $A\in\Pi_{\theta}$, then  
\[\cos \left( \theta  \right)\left\| A \right\|\le \left\| \mathfrak RA \right\|.\]
Therefore, Corollary \ref{cor_R_t} implies
\[\left( 2\sqrt{{{\left( 1-t \right)}^{3}}}\cos \left( \theta  \right) \right)\left\| A \right\|+\left| \,\left\| \sqrt{{{\left( 1-t \right)}^{3}}}\mathfrak RA+\sqrt{{{t}^{3}}}\mathfrak IA \right\|-\left\| \sqrt{{{\left( 1-t \right)}^{3}}}\mathfrak RA-\sqrt{{{t}^{3}}}\mathfrak IA \right\| \,\right|\le {{\omega }_{t}}\left( A \right).\]
In particular,
\[\frac{\cos \left( \theta  \right)}{\sqrt{2}}\left\| A \right\|+\frac{1}{2\sqrt{2}}\left| \,\left\| \mathfrak RA+\mathfrak IA \right\|-\left\| \mathfrak RA-\mathfrak IA \right\| \,\right|\le \omega \left( A \right),\]
when $A\in\Pi_{\theta}.$

Consequently, 
 if  $0\le \theta \le {\pi }/{4}$, we have $\cos\theta\geq\frac{1}{\sqrt{2}}$, and hence
\[\begin{aligned}
   \frac{1}{2}\left\| A \right\|&\le \frac{1}{2}\left\| A \right\|+\frac{1}{2\sqrt{2}}\left| \,\left\| \mathfrak RA+\mathfrak IA \right\|-\left\| \mathfrak RA-\mathfrak IA \right\| \, \right| \\ 
 & \le \frac{\cos \left( \theta  \right)}{\sqrt{2}}\left\| A \right\|+\frac{1}{2\sqrt{2}}\left| \,\left\| \mathfrak RA+\mathfrak IA \right\|-\left\| \mathfrak RA-\mathfrak IA \right\| \,\right| \\ 
 & \le \omega \left( A \right);\;A\in\Pi_{\theta}, 0\leq \theta\leq \frac{\pi}{4}.  
\end{aligned}\]
This provides a considerable refinement of the inequality $\frac{1}{2}\|A\|\leq\omega(A).$ It is also worthwhile to mention here that, in this case, our result improves \cite[Proposition 3.1]{2}.
\end{remark}

Although different, the following result provides a possible reversed version of the well known power inequality $\omega \left( {{A}^{2}} \right)\le {{\omega }^{2}}\left( A \right)$, for any $A\in\mathcal{B}(\mathcal{H})$ (see \cite[Theorem 2.1-1]{1}).
\begin{theorem}
Let $A\in\Pi_{\alpha}$. Then
	\[\left( 1-2{{\tan }^{2}}\left( \alpha  \right) \right){{\omega }^{2}}\left( A \right)\le \omega \left( {{A}^{2}} \right).\]
\end{theorem}
\begin{proof}
We clearly have
	\[\mathfrak R{{A}^{2}}+2{{\left( \mathfrak IA \right)}^{2}}=\frac{{{\left| A \right|}^{2}}+{{\left| {{A}^{*}} \right|}^{2}}}{2}.\]
By the mixed Cauchy-Schwarz inequality \cite[pp. 75--76]{hal}, we have for any unit vector $x\in \mathcal H$, 
	\[\begin{aligned}
   {{\left| \left\langle Ax,x \right\rangle  \right|}^{2}}&\le \left\langle \left| A \right|x,x \right\rangle \left\langle \left| {{A}^{*}} \right|x,x \right\rangle  \\ 
 & \le \frac{{{\left\langle \left| A \right|x,x \right\rangle }^{2}}+{{\left\langle \left| {{A}^{*}} \right|x,x \right\rangle }^{2}}}{2}\\
 & \qquad \text{(by the arithmetic-geometric mean inequality) }\\ 
 & \le \frac{\left\langle {{\left| A \right|}^{2}}x,x \right\rangle +\left\langle {{\left| {{A}^{*}} \right|}^{2}}x,x \right\rangle }{2} \\ 
 & =\left\langle \left( \frac{{{\left| A \right|}^{2}}+{{\left| {{A}^{*}} \right|}^{2}}}{2} \right)x,x \right\rangle .  
\end{aligned}\]
Namely,
	\[{{\left| \left\langle Ax,x \right\rangle  \right|}^{2}}\le \left\langle \mathfrak R{{A}^{2}}x,x \right\rangle +\left\langle 2{{\left( \mathfrak IA \right)}^{2}}x,x \right\rangle .\]
This implies by taking supremum over $x\in \mathcal H$ with $\left\| x \right\|=1$,
	\[{{\omega }}\left( A \right)^{2}\le \left\| \mathfrak R{{A}^{2}} \right\|+2{{\left\| \mathfrak IA \right\|}^{2}}.\]
Now, according to the assumption $W\left( A \right)\subset {{S}_{\alpha }}$, we infer that
	\[\left\| \mathfrak IA \right\|\le \tan \left( \alpha  \right)\left\| \mathfrak RA \right\|.\]
In this case we get
	\[\begin{aligned}
  {{\omega }}\left( A \right)^{2}& \le \left\| \mathfrak R{{A}^{2}} \right\|+2{{\left\| \mathfrak IA \right\|}^{2}} \\ 
 & \le \left\| \mathfrak R{{A}^{2}} \right\|+2{{\tan }^{2}}\left( \alpha  \right){{\left\| \mathfrak RA \right\|}^{2}} \\ 
 & \le \left\| \mathfrak R{{A}^{2}} \right\|+2{{\tan }^{2}}\left( \alpha  \right){{\omega }^{2}}\left( A \right).  
\end{aligned}\]
Namely,
	\[\left( 1-2{{\tan }^{2}}\left( \alpha  \right) \right){{\omega }}\left( A \right)^{2}\le \left\| \mathfrak R{{A}^{2}} \right\|.\]
Since
\[\left\| \mathfrak R{{A}^{2}} \right\|=\omega \left( \mathfrak R{{A}^{2}} \right)\le \omega \left( {{A}^{2}} \right),\]
we deduce the desired result.
\end{proof}
In this theorem, if $\alpha=0,$ the operator $A$ is positive, and we have $\omega(A^2)=\omega(A)^2$. The significance of this theorem is evident when $1-2\tan^2\alpha>0$ or $\alpha<\tan^{-1}\left(\frac{1}{\sqrt{2}}\right).$

\noindent{\tiny (C. Conde)  Instituto de Ciencias, Universidad Nacional de General Sarmiento  and  Consejo Nacional de Investigaciones Cient\'ificas y Tecnicas, Argentina}

\noindent{\tiny \textit{E-mail address:} cconde@campus.ungs.edu.ar}

\vskip 0.3 true cm

\noindent{\tiny (M. Sababheh) Department of Basic Sciences, Princess Sumaya University for Technology, Amman, Jordan}
	
\noindent	{\tiny\textit{E-mail address:} sababheh@psut.edu.jo}

\vskip 0.3 true cm

\noindent{\tiny (H. R. Moradi) Department of Mathematics, Payame Noor University (PNU), P.O. Box, 19395-4697, Tehran, Iran
	
\noindent	\textit{E-mail address:} hrmoradi@mshdiau.ac.ir}

\end{document}